\documentclass[a4paper,fleqn]{cas-sc}

\usepackage[numbers,sort&compress]{natbib}

\usepackage{amsmath,amssymb}
\usepackage{amsthm}

\usepackage[capitalise,nameinlink,noabbrev]{cleveref}

\usepackage{tikz}
\tikzset{
every node/.style={draw,circle,inner sep=2pt}
}

\usepackage[normalem]{ulem}
\usepackage{cancel}
\usepackage{xcolor}

\def\tsc#1{\csdef{#1}{\textsc{\lowercase{#1}}\xspace}}
\tsc{WGM}
\tsc{QE}

\newtheorem{theorem}{Theorem}[section]
\newtheorem{lemma}[theorem]{Lemma}
\newtheorem{proposition}[theorem]{Proposition}
\newtheorem{corollary}[theorem]{Corollary}

\theoremstyle{definition}
\newtheorem{definition}[theorem]{Definition}

\newtheorem{remark}[theorem]{Remark}
\newtheorem{example}[theorem]{Example}

\newcommand{\trans}{^\top}

\newcommand{\bone}{\mathbf{1}}

\newcommand{\bx}{\boldsymbol{\eta}}  

\newcommand{\bz}{\mathbf{z}}
\newcommand{\bu}{\mathbf{u}}
\newcommand{\bv}{\mathbf{v}}

\newcommand{\bZ}{\mathbf{Z}}
\newcommand{\bA}{\mathbf{A}}
\newcommand{\bR}{\mathbf{R}}

\newcommand{\bI}{\mathbf{I}}
\newcommand{\bDtwo}{\mathbf{D}_2}
\newcommand{\bX}{\mathbf{X}}
\newcommand{\bM}{\mathbf{M}}
\newcommand{\bT}{\mathbf{T}}
\newcommand{\bQ}{\mathbf{Q}}
\newcommand{\bP}{\boldsymbol{\Phi}}
\newcommand{\bOmega}{\boldsymbol{\Omega}}
\newcommand{\bpsi}{\boldsymbol{\psi}}
\newcommand{\bphi}{\boldsymbol{\phi}}
\newcommand{\bE}{\mathbf{E}}
\newcommand{\bK}{\mathbf{K}}
\newcommand{\bW}{\mathbf{W}}
\newcommand{\balpha}{\boldsymbol{\alpha}}

\newcommand{\rank}{\operatorname{rank}}

\begin{document}

\let\WriteBookmarks\relax
\def\floatpagepagefraction{1}
\def\textpagefraction{.001}

\shorttitle{Spline Parity and Oscillation}

\shortauthors{Huang et al.}


\title[mode=title]{Sturm--Liouville-Type Parity and Oscillation of a Cubic Spline Eigenbasis}


\author[1,2]{Shih-Hao Huang}


\credit{}

\affiliation[1]{%
    organization={Department of Mathematics, National Central University},
    city={Taoyuan},
    postcode={320317},
    country={Taiwan}
}

\affiliation[2]{%
    organization={Graduate Institute of Statistics, National Central University},
    city={Taoyuan},
    postcode={320317},
    country={Taiwan}
}

\author[3]{Jephian C.-H. Lin}


\credit{}

\affiliation[3]{%
    organization={Department of Applied Mathematics, National Yang Ming Chiao Tung University},
    city={Hsinchu},
    postcode={300093},
    country={Taiwan}
}

\author[4]{ShengLi Tzeng}[orcid= 0000-0003-3879-9420]

\cormark[1]

\ead{slt.cmu@gmail.com}

\credit{}

\affiliation[4]{%
    organization={Department of Applied Mathematics and Graduate Institute of Statistics,  National Chung
Hsing University},
    city={ Taichung},
    postcode={402202},
    country={Taiwan}
}

\author[5]{Tzu-Lun Yuan}


\credit{}

\affiliation[5]{%
    organization={Department of Statistics, Tunghai University},
    city={ Taichung},
    postcode={40704},
    country={Taiwan}
}


\cortext[1]{Corresponding author.}


\begin{abstract}
We study the eigen-structure of the  kernel matrix arising from cubic smoothing splines on equally spaced knots. Using  matrix-theoretic methods, we show that its positive eigenvalues are simple, that the associated eigenvectors  alternate between even and odd, and that the eigenvector for the $k$th largest positive eigenvalue has exactly $k+1$ sign changes. 
Building on established results on oscillatory and centrosymmetric matrices, we provide a more explicit proof of the known oscillation property of the eigenvectors and establish their alternating parity. 
We then verify, by a separate argument, that both properties extend to  the associated continuous cubic spline eigenfunctions, showing that the resulting spline basis exhibits   both  an alternating parity structure and an oscillation pattern analogous to that of classical Sturm--Liouville eigenfunctions. 
\end{abstract}




\begin{keywords}
Oscillatory matrices\\
Totally nonnegative matrices\\
Demmler--Reinsch basis\\
Parity of eigenvectors\\
Spline smoothing
\end{keywords}

\maketitle

\section{Introduction}

Sturm--Liouville eigenfunctions, including the cosine and sine bases and classical orthogonal polynomials, provide canonical examples of basis functions with rich oscillatory structures. Two properties make them particularly useful for approximation. First, they are naturally ordered by oscillatory complexity. The Sturm oscillation theorem establishes that successive eigenfunctions (ordered by increasing eigenvalue) possess an increasing number of interior zeros 

\cite{Sturm1836,Zettl2005}. The classical orthogonal polynomials 
share this property, which follows from orthogonality alone without the Sturm oscillation theorem \cite{Szego1975}. Second, in the symmetric cases these bases alternate in parity about the midpoint of the interval. This can be checked directly for the cosine and sine bases, and for the orthogonal polynomial $p_n$ of degree $n$ with an even weight function it takes the form $p_n(-x)=(-1)^n p_n(x)$. The alternating parity yields a direct decomposition of the approximation space into symmetric and antisymmetric subspaces \cite{Szego1975}. This parity decomposition allows the even and odd components of a target function to be approximated independently, a feature that plays a central role in Fourier expansions and classical spectral methods on symmetric domains \cite{Boyd2001,Trefethen2019}. These bases are thus both hierarchically organized by oscillation and adapted to the symmetry of the underlying problem.

Motivated by these advantages, it is natural to seek spline bases that possess analogous oscillation and parity properties. Such bases would combine the flexibility of splines with an intrinsic ordering by oscillatory complexity and a natural decomposition into even and odd subspaces. A natural candidate is the Demmler--Reinsch basis arising from the cubic smoothing-spline roughness penalty \cite{DR75}. The associated eigendecomposition, however, is not a classical Sturm--Liouville equation, so the oscillation and parity properties cannot be read off from the classical theory and must be established by other means.

In this article, we show that a set of linear combinations of natural cubic splines consists entirely of functions that are either even or odd. These combinations are derived from the eigendecomposition of the projected kernel matrix of smoothing splines  on equally spaced knots.  While the periodic setting leads to band circulant matrices (analyzed by \cite{Dubeau1985}), the present work focuses on the non-periodic case. Since the even and odd eigenvectors (Definition~\ref{def:even-odd}) of this matrix appear alternately, the corresponding eigenfunctions inherit the same alternating even-odd structure. The alternation is also convenient in practice: since the parity of each eigenfunction is determined by its index, symmetry can be imposed or tested simply by retaining or examining the even or the odd members, with each group remaining ordered by oscillation. Our analysis largely uses results on oscillatory matrices from the book \cite{GK02}, and on centrosymmetric matrices from \cite{CB76}.

We begin by introducing the matrices that will be used throughout the proof.

\begin{definition}
Define the following matrices:
\[
\bOmega=\bQ\bP\bQ,\qquad
\bP=\bigl[\,|i-j|^3\,\bigr],\quad
\bQ=\bI-\bX(\bX^{\trans}\bX)^{-1}\bX^{\trans}, \quad
\bX=\begin{bmatrix}1&1&\cdots&1\\1&2&\cdots&n\end{bmatrix}^{\trans},
\]
where $\bOmega$, $\bP$, and $\bQ$ are $n\times n$.
\label{def:Omega}
\end{definition}

We assume throughout  that \(n\ge 3\). The matrix $\bP$ is the  kernel matrix generated by the cubic
radial function $\phi(r)=|r|^3$ evaluated pairwise at the equally spaced knots $1,\ldots,n$,
and $\bQ$ is the projection onto the orthogonal complement of the linear
polynomials spanned by the columns of $\bX$; this projection  restores positive
semidefiniteness.  Moreover, $\bOmega=\bQ\bP\bQ$
is, up to a normalizing constant, exactly the roughness-penalty matrix
 of the cubic
smoothing spline at these knots in the kernel-coefficient coordinates;  in function-value coordinates,  the penalty matrix is,  up to the same factor,  the pseudoinverse of $\bOmega$ \cite{RC67,DR75}. The eigenvectors of $\bOmega$ are themselves the coefficient vectors of the Demmler--Reinsch spline eigenfunctions in this kernel representation (Section \ref{sec:spline-parity} - \ref{sec:spline-osc}), since the two matrices share eigenvectors, with the order of the eigenvalues reversed.   This connection links our spectral results for 
$\bOmega$  to the structure of the associated  spline basis.

Our main result (Theorem~\ref{thm:main}) shows that $\bOmega$ is positive
semidefinite with nullity $2$, that every positive eigenvalue is simple, that every eigenvector  for a positive
eigenvalue is  either even or odd, and that the eigenvector for
the $k$th largest positive eigenvalue    has exactly $k+1$ sign changes  and is even (respectively, odd) when $k+1$ is an even (odd) number. These spectral properties are then transferred to the associated cubic spline eigenfunctions: the spline eigenfunctions inherit the same alternating even--odd parity (Theorem~\ref{thm:spline-parity}) and possess exactly $k+1$ simple zeros in the open interval  (Theorem~\ref{thm:spline-oscillation}).

This construction has been employed in previous works, including \cite{TH2018,yuan2020}. Those studies, however, focused on spline spaces over nonuniform knot configurations, where the parity structure arising in the uniformly spaced case was neither observed nor investigated. Moreover, several of the structural properties established in this paper, namely the semidefiniteness, nullity,   and simplicity of the positive eigenvalues, have appeared previously in the literature; see, for example, \cite{RC67,DR75}.  The matrix-theoretic arguments developed below rely on classical total positivity and oscillatory matrix theory, whose roots go back to \cite{SW53}.  We revisit the proof in \cite{DR75} by specializing its oscillatory-matrix argument to equally spaced knots; this makes the role of the spline kernel matrix explicit and yields the alternating parity of the eigenvectors. We then extend the resulting oscillation and parity properties to the continuous cubic spline eigenfunctions.

\section{Preliminaries} \label{sec:prelim}

In this section,   we introduce three classes of matrices that underlie the proofs of the results developed later:  centrosymmetric, totally nonnegative, and oscillatory matrices, and recall their basic properties.

\begin{definition}
\label{def:even-odd}
Let $\bZ$ be the matrix whose $i,j$-entry is $1$ whenever
$i+j=m+1$ and zero otherwise. A vector $\bv\in\mathbb{R}^m$ is
\emph{even} if $\bZ\bv=\bv$ and \emph{odd} if $\bZ\bv=-\bv$.
An $m\times m$ matrix $\bA$ is \emph{centrosymmetric} if
\[
\bZ\bA\bZ=\bA.
\]
\end{definition}

In \cite{CB76}, it was shown that   a symmetric centrosymmetric matrix admits an eigenbasis consisting entirely of even or odd vectors. Although a non-symmetric
matrix need not be diagonalizable,   each eigenspace of a centrosymmetric matrix is invariant under $\bZ$ and therefore admits a basis of even and/or odd vectors; in particular, every eigenvector belonging to a simple eigenvalue is either even or odd. Throughout the paper, ``eigenvector'' refers to a right eigenvector for nonsymmetric matrices.  We include the proof for
completeness.

\begin{proposition}
\label{prop:csym}
 Let $\lambda$ be a simple real eigenvalue of a centrosymmetric matrix $\bA$. Then an eigenvector associated with  $\lambda$ is either even or odd.
\end{proposition}

\begin{proof}
Let $\mathbb{R}_{\rm even}$ and $\mathbb{R}_{\rm odd}$ denote the subspaces
of $\mathbb{R}^m$ consisting of all even vectors and all odd vectors,
respectively. Then
$\dim(\mathbb{R}_{\rm even})=\lceil m/2\rceil$,
$\dim(\mathbb{R}_{\rm odd})=\lfloor m/2\rfloor$, and
$\mathbb{R}^m=\mathbb{R}_{\rm even}\oplus\mathbb{R}_{\rm odd}$.
We show both subspaces are $\bA$-invariant. If $\bv\in\mathbb{R}_{\rm even}$,
then $\bZ\bv=\bv$ gives
$\bZ\bA\bv=\bZ\bA\bZ\bZ\bv=\bA(\bZ\bv)=\bA\bv$, so $\bA\bv\in\mathbb{R}_{\rm even}$.
Similarly, if $\bv\in\mathbb{R}_{\rm odd}$, then $\bZ\bv=-\bv$ gives
$\bZ\bA\bv=\bZ\bA\bZ\bZ\bv=\bA(\bZ\bv)=-\bA\bv$, so $\bA\bv\in\mathbb{R}_{\rm odd}$.

Let $\bv$ be an eigenvector of $\bA$ for $\lambda$, and write
\[
\bv=\bv_e+\bv_o,
\qquad
\bv_e\in\mathbb{R}_{\rm even},\quad
\bv_o\in\mathbb{R}_{\rm odd},
\]
using the direct sum above. Since both subspaces are $\bA$-invariant,
\[
\bA\bv_e\in\mathbb{R}_{\rm even},
\qquad
\bA\bv_o\in\mathbb{R}_{\rm odd}.
\]
On the other hand, since $\bA\bv=\lambda\bv$,
\[
\bA\bv
=\lambda\bv
=\lambda\bv_e+\lambda\bv_o.
\]
The uniqueness of the decomposition
$\mathbb{R}^m=\mathbb{R}_{\rm even}\oplus\mathbb{R}_{\rm odd}$ therefore implies
\[
\bA\bv_e=\lambda\bv_e,
\qquad
\bA\bv_o=\lambda\bv_o.
\]
If both $\bv_e\ne\mathbf0$ and $\bv_o\ne\mathbf0$, they would be linearly independent
eigenvectors of $\bA$ corresponding to the same eigenvalue $\lambda$, contradicting the
simplicity of $\lambda$. Hence exactly one of $\bv_e$ and $\bv_o$ is zero, so $\bv$ is
either even or odd.
\end{proof}

Recall that a minor of a matrix is the determinant of one of its square
submatrices, as in the definition below.

\begin{definition}[Definition~4, Section II.2 of \cite{GK02}]
\label{def:tn}
A matrix is called \emph{totally nonnegative} (respectively,
\emph{totally positive}) if every minor is nonnegative (respectively,
positive).  
\end{definition}

Since every minor of a submatrix is also a minor of the original matrix, every submatrix of a totally nonnegative (respectively,
totally positive) matrix is again totally nonnegative (respectively,
totally positive).

\begin{proposition}[Proposition~1$^\circ$, Section II.2 of \cite{GK02}]
\label{prop:tnprod}
The product of two totally nonnegative matrices is totally nonnegative.
\end{proposition}

Let $\bR$ denote the diagonal matrix whose $i$th diagonal entry equals
$(-1)^i$. Its size will always be clear from the context.

\begin{proposition}[Proposition~3$^\circ$, Section II.2 of \cite{GK02}]
\label{prop:tninv}
Let $\bA$ be an invertible matrix. Then $\bA^{-1}$ is totally nonnegative if and
only if $\bR\bA\bR$ is totally nonnegative.
\end{proposition}

\begin{definition}[Definition~6, Section II.2 of \cite{GK02}]
\label{def:os}
A matrix is called an \emph{oscillatory matrix} if it is totally
nonnegative and some positive power of it is totally positive.
\end{definition}

\begin{theorem}[Theorem~10, Section II.7 of \cite{GK02}]
\label{thm:oscriterion}
Let $\bA=[a_{i,j}] $ be a totally nonnegative matrix.
Then $\bA$ is oscillatory if and only if

\begin{enumerate}
\item $\bA$ is invertible, and
\item $a_{i,j}>0$ whenever $|i-j|=1$.
\end{enumerate}
\end{theorem}

\begin{proposition}[Proposition~6$^\circ$, Section II.7 of \cite{GK02}]
\label{prop:ostnos}
The product of an oscillatory matrix and a nonsingular totally nonnegative
matrix is oscillatory.
\end{proposition}

Oscillatory matrices possess many remarkable properties concerning sign
changes and nodal locations of eigenvectors. Here we only need the result on
sign changes.

\begin{definition}
\label{def:S_pm}
Let $\bu=(u_1,\ldots,u_m)$ be a vector. Assigning signs to zero entries
allows us to define the minimum possible number of sign changes, denoted by
$S^{-}({\bu})$, and the maximum possible number of sign changes, denoted by
$S^{+}({\bu})$. When $S^{-}({\bu})=S^{+}({\bu})$, their common value is denoted
by $S({\bu})$.
\end{definition}

For example, $\bu_1=(1,0,1)$ and $\bu_2=(1,0,-1)$ satisfy
$S^{-}({\bu_1})=0$, $S^{+}({\bu_1})=2$, and $S({\bu_2})=1$. More precisely,
$S^{-}({\bu})=S^{+}({\bu})$ if and only if $u_1u_m\neq0$ and
$u_{i-1}u_{i+1}<0$ whenever $u_i=0$ for $i=2,\ldots,m-1$.

\begin{theorem}[Theorem~6, Section II.5 of \cite{GK02}]
\label{thm:osspec}
Let $\bA$ be an oscillatory matrix. Then
\begin{enumerate}
\item every eigenvalue of $\bA$ is simple, real, and positive;
\item if $\bu_k$ is an eigenvector corresponding to the $k$th largest
eigenvalue, then $S({\bu_k})=k-1$.
\end{enumerate}
\end{theorem}

\begin{definition}
A \emph{Jacobi matrix} is a real square matrix
$\bA=[a_{i,j}]$
such that
$a_{i,j}=0$
whenever
$|i-j|\ge2$.
Throughout this paper we adopt the convention of
\cite[Section II.1]{GK02}; in contrast to some authors,
we do not require Jacobi matrices to be symmetric or irreducible.
\end{definition}

Recall that the leading principal minors of a matrix are the determinants of the principal submatrices induced on $\{1, \ldots, k\}$ for some k.

\begin{proposition}[Item 6e, Section II.3 of \cite{GK02}]
\label{prop:osjacobi}
A Jacobi matrix
$\bA=[a_{i,j}]$
is totally nonnegative if 

\begin{enumerate}
\item $a_{i,j}\ge0$ whenever $|i-j|=1$, and
\item every leading principal minor is positive.
\end{enumerate}
\end{proposition}

\begin{proposition}\label{prop:jacobi_oscillatory}
Let $\bA=[a_{ij}]$ be a symmetric $m \times m$ Jacobi matrix satisfying
\[
  a_{ij}>0\ (|i-j|=1),\qquad
a_{ii}>\sum_{j\ne i}|a_{ij}|,\quad i=1,\ldots,m.
\]
Then $\bA$ is oscillatory.
\end{proposition}

\begin{proof}
For symmetric $\bA$, the Gershgorin circle theorem
\cite[Sec.~6.1]{HJ13} 
and the strict diagonal dominance
$a_{ii}>\sum_{j\ne i}|a_{ij}|$ yield positive eigenvalues, hence
positive definiteness. Proposition~\ref{prop:osjacobi} and
Theorem~\ref{thm:oscriterion} imply that $\bA$ is oscillatory.
\end{proof}

\begin{example}
With $\sigma>0$, the matrix $\bA_\sigma= \bigl[e^{-\sigma|i-j|}\bigr] $ 
is oscillatory by \cite[Page~79]{GK02}.
More generally, if $x_1<\cdots<x_n,$
then $\bA_\sigma= \bigl[e^{-\sigma|x_i-x_j|}\bigr]$
is also oscillatory.
Since symmetric oscillatory matrices are positive definite (\ref{thm:osspec}), this shows that
the exponential kernel $e^{-\sigma|x-y|}$, common in statistics, is strictly
positive definite on any finite set of distinct points.
\end{example}

\begin{lemma}[Equation~(4.1) of \cite{fallat2011totally}]
\label{lem:signflip}
For every $\bu\in\mathbb{R}^\ell$ with $\ell \ge 1$,
\[
S^{+}(\bu)+S^{-}(\bR\bu)=\ell-1.
\]
\end{lemma}

\begin{corollary}
\label{cor:signflip}
For every $\ell \ge 1$ and every $\bu\in\mathbb{R}^\ell$,
\[
S^{+}(\bR\bu)=\ell-1-S^{-}(\bu).
\]
\end{corollary}
\begin{proof}
Applying Lemma~\ref{lem:signflip} to $\bR\bu$ and using $\bR^2=\bI$, we obtain
\[
S^{+}(\bR\bu)+S^{-}(\bu)=\ell-1.
\]
\end{proof}

\begin{lemma}[{\cite[Theorem~4.2.2]{fallat2011totally}}]
\label{lem:weakTN}
If $\bA$ is an $m\times \ell$ totally nonnegative matrix, not necessarily square, then
\[
S^{-}(\bA\bu)\le S^{-}(\bu)\qquad\text{for all } \bu\in\mathbb{R}^{\ell}.
\]
\end{lemma}

\begin{lemma}[{\cite[Theorem~4.2.5]{fallat2011totally}}]
\label{lem:InTN}
If $\bA$ is an $\ell \times \ell$ invertible totally nonnegative matrix, then for all
$\bu\in\mathbb{R}^{\ell}$,
\[
S^{-}(\bA\bu)\le S^{-}(\bu)
\qquad\text{and}\qquad
S^{+}(\bA\bu)\le S^{+}(\bu).
\]
\end{lemma}

We isolate the bidiagonal building blocks arising in certain structurally special rectangular matrices. The matrix  $\bDtwo$ used later in Definition~\ref{def:D2} provides one such example.
For $m\ge 2$, let $\bE^{(m)}:\mathbb{R}^{m-1}\to\mathbb{R}^{m}$ denote the matrix
\[
(\bE^{(m)}\bu)_1=u_1,\qquad
(\bE^{(m)}\bu)_i=u_{i-1}+u_i\ \ (2\le i\le m-1),\qquad
(\bE^{(m)}\bu)_m=u_{m-1}.
\]

\begin{lemma}
\label{lem:einj}
For every $m\ge2$, $\bE^{(m)}\bu=\mathbf{0}$ implies $\bu=\mathbf{0}$.
\end{lemma}
\begin{proof}
The first coordinate of $\bE^{(m)}\bu$ gives $u_1=0$. 
For $2\le i\le m-1$, the $i$th
coordinate gives $u_{i-1}+u_i=0$, i.e.\ $u_i=-u_{i-1}$; starting from $u_1=0$, induction
forces $u_2=\cdots=u_{m-1}=0$. Hence $\bu=\mathbf{0}$.
\end{proof}

\begin{lemma}
\label{lem:Evm}
For every $m\ge 2$ and every $\bu\in\mathbb{R}^{m-1}\setminus\{\mathbf{0}\}$,
\[
S^{+}\!\bigl(\bE^{(m)}\bu\bigr)\le S^{+}(\bu).
\]
\end{lemma}
\begin{proof}
We induct on $m$. Two elementary facts about $S^{+}$, immediate from
the definition of $S^+$ (cf.\ \S4.1 of
\cite{fallat2011totally}), are used repeatedly: for $\bv\in\mathbb R^\ell$, appending a trailing zero raises
$S^{+}$ by exactly one, $S^{+}(\bv,0)=S^{+}(\bv)+1$; and appending a copy of a nonzero
last entry leaves $S^{+}$ unchanged, $S^{+}(\bv,v_\ell)=S^{+}(\bv)$ when $v_\ell\ne0$.

For the base case $m=2$, $\bu=(u_1)\ne0$ and $\bE^{(2)}\bu=(u_1,u_1)^{\trans}$; both
have $S^{+}=0$.

For $2\le i\le m$ and $c\in\mathbb{R}$, write $L_i(c)$ for the $m\times m$ elementary
bidiagonal (EB) matrix that differs from the identity only in the $(i,i-1)$ entry, which
equals $c$; equivalently, 
\[
L_i(c)
=
\left(
\delta_{jk}
+
c\,\delta_{ji}\delta_{k,i-1}
\right)_{j,k=1}^m.
\]
It follows from \cite[Lemma~4.2.1(b)]{fallat2011totally} that
\[
S^{+}(L_i(1)\bz)\le S^{+}(\bz)\qquad\text{for all }\bz\in\mathbb{R}^m.
\]

Let $\mathbf{B}$ denote the $m\times m$ lower-bidiagonal matrix with every diagonal and
subdiagonal entry equal to $1$,
\[
\mathbf{B}=\begin{pmatrix}1&&&\\1&1&&\\&\ddots&\ddots&\\&&1&1\end{pmatrix},
\qquad\text{i.e.}\qquad
(\mathbf{Bz})_1=z_1,\quad (\mathbf{Bz})_i=z_{i-1}+z_i\ \ (2\le i\le m).
\]
A direct computation gives 
\[
\mathbf{B}\,(\bu^{\trans},0)^{\trans}= \bigl(u_1,\ u_1+u_2,\ldots,u_{m-2}+u_{m-1},\ u_{m-1}\bigr)^{\trans}
=\bE^{(m)}\bu.
\]
Moreover, $\mathbf{B}$ factors as
\[
\mathbf{B}=L_2(1)L_3(1)\cdots L_m(1).
\]
The order of the factors is essential: reversing them would produce a lower-triangular matrix rather than the desired bidiagonal matrix.

If $u_{m-1}\ne0$ (the last entry of $\bu$ is nonzero), the rightmost factor $L_m(1)$ of
$\mathbf{B}$ adds row $m-1$ into row $m$ and carries $(\bu^{\trans},0)^{\trans}$ to
$\bu_{*}:= (\bu^{\trans},u_{m-1})^{\trans}$, and $S^{+}(\bu_{*})=S^{+}(\bu)$ by the second elementary
fact above. Every remaining EB factor of $\mathbf{B}$ can only weakly decrease $S^{+}$
(\cite[Lemma~4.2.1(b)]{fallat2011totally}), so
$S^{+}(\bE^{(m)}\bu)\le S^{+}(\bu_{*})=S^{+}(\bu)$.

If $u_{m-1}=0$, write $\bu=(\widehat{\bu},0)$ with $\widehat{\bu}\in\mathbb{R}^{m-2}$,
$\widehat{\bu}\ne\mathbf0$ (so $m\ge3$). Then
\[
\bigl(\bE^{(m)}\bu\bigr)_1=u_1=\bigl(\bE^{(m-1)}\widehat{\bu}\bigr)_1,
\qquad
\bigl(\bE^{(m)}\bu\bigr)_i=u_{i-1}+u_i=\bigl(\bE^{(m-1)}\widehat{\bu}\bigr)_i
\quad(2\le i\le m-1),
\]
where for $i=m-1$ we use $u_{m-1}=0$ so that $u_{m-2}+u_{m-1}=\widehat u_{m-2}$,
and $\bigl(\bE^{(m)}\bu\bigr)_m=u_{m-1}=0$. Thus
$\bE^{(m)}\bu=\bigl(\bE^{(m-1)}\widehat{\bu},\,0\bigr)$. By the induction hypothesis,
$S^{+}(\bE^{(m-1)}\widehat{\bu})\le S^{+}(\widehat{\bu})$, and adding a trailing zero to
both sides  gives
$S^{+}(\bE^{(m)}\bu)=S^{+}(\bE^{(m-1)}\widehat{\bu})+1\le S^{+}(\widehat{\bu})+1=S^{+}(\bu)$.
\end{proof}

\begin{remark}
\label{rem:ZM}
Note that the bidiagonal factorization is a powerful technique in study of totally nonnegative matrices; see, e.g., \cite{Fallat2001,Adm2018}. 
The proof of Lemma~\ref{lem:Evm} uses the same successive elementary bidiagonal (SEB) factorization technique as \cite{ZM2020}. In both settings, a matrix is decomposed into  EB factors  from \cite[\S2.1]{fallat2011totally}, which are then analyzed one factor at a time. In \cite{ZM2020}, this approach is combined with a planar-network representation to track the positivity of corner minors and determine when a power of an oscillatory matrix becomes totally positive. Here, we instead track the sign variation $S^{+}$.  The same technique could also be used to determine the smallest exponent $p$ such that a given oscillatory matrix raised to the $p$th power is totally positive, although this is not needed below.  
\end{remark}

\section{Main Theorem} \label{sec:main}

\begin{definition}\label{def:D2}
  Define the $n \times (n-2)$ \emph{twice-differencing matrix} $\bDtwo$ whose
  $j$th column is
 \[ 
  \left[ \mathbf{0}_{j-1}\trans, 1, -2, 1, \mathbf{0}_{n-j-2}\trans \right] \trans, 
  \qquad j = 1,\ldots, n-2.
\]
  Set $\bM = \bDtwo\trans \bOmega \bDtwo$ and $\bT = \bDtwo\trans \bDtwo$, both square
  matrices of order $n-2$.
\end{definition}

\begin{lemma}
\label{lem:D2E}
$|\bDtwo|=\bE^{(n)}\bE^{(n-1)}$, where $|\bDtwo|$ is the matrix of entrywise absolute values of $\bDtwo$.
\end{lemma}
\begin{proof}
Immediate from Definition~\ref{def:D2}: each column of $\bE^{(n-1)}$ places a $1$ in two
consecutive coordinates, and applying $\bE^{(n)}$ to the result reproduces, entrywise, the
$(1,2,1)$ band of $|\bDtwo|$.
\end{proof}

\begin{corollary}
\label{cor:D2vm}
For every $\bu\in\mathbb{R}^{n-2}\setminus\{\mathbf 0\}$,
\[
S^{+}\bigl(|\bDtwo|\bu\bigr)\le S^{+}(\bu).
\]
\end{corollary}
\begin{proof}
By Lemma~\ref{lem:einj}, $\bE^{(n-1)}\bu\ne\mathbf0$, so Lemma~\ref{lem:Evm} may be applied
twice: $S^{+}(\bE^{(n-1)}\bu)\le S^{+}(\bu)$, and then, since $\bE^{(n-1)}\bu\ne\mathbf0$,
$S^{+}\bigl(\bE^{(n)}(\bE^{(n-1)}\bu)\bigr)\le S^{+}(\bE^{(n-1)}\bu)$. Composing the two
inequalities and using Lemma~\ref{lem:D2E} gives the claim.
\end{proof}

\begin{proposition}
\label{prop:D2_part}
The columns of $\bX$ are orthogonal to the columns of $\bDtwo$.
Moreover, $ \begin{bmatrix} \bX&\bDtwo \end{bmatrix}$ is invertible.
\end{proposition}

\begin{proof}
The orthogonality follows from a direct computation. Furthermore, both $\bX$
and $\bDtwo$ have full column rank, so both $\bX^{\trans}\bX$ and
$\bDtwo^{\trans}\bDtwo$ are invertible, and consequently
\begin{equation}
\label{eq:X-D2-inv}
\begin{bmatrix}
\bX&\bDtwo
\end{bmatrix}^{-1}
=
\begin{bmatrix}
\left(\bX^{\trans}\bX\right)^{-1}&\mathbf{O}\\
\mathbf{O} &\left(\bDtwo^{\trans}\bDtwo\right)^{-1}
\end{bmatrix}
\begin{bmatrix}
\bX^{\trans}\\
\bDtwo^{\trans}
\end{bmatrix},
\end{equation}
which proves the claim.
\end{proof}

\begin{lemma}
\label{lem:asim}
The matrix $\bOmega$ is similar to $\mathbf{O}_2\oplus(\bT^{-1}\bM)$,
and therefore they have the same spectrum. Moreover, if $\bx$ is an   eigenvector of $\bT^{-1}\bM$, then
$\bDtwo\bx$
is an eigenvector of $\bOmega$ corresponding to the same eigenvalue.
\end{lemma}

\begin{proof}
By definition, $\bOmega\bX=\mathbf{O}$. Here and below, $\mathbf{O}$
 denotes a zero matrix of appropriate size. Using \eqref{eq:X-D2-inv}, we obtain
\[
\begin{aligned}
\begin{bmatrix}
\bX&\bDtwo
\end{bmatrix}^{-1}
\bOmega
\begin{bmatrix}
\bX&\bDtwo
\end{bmatrix}
&=
\begin{bmatrix}
\left(\bX^{\trans}\bX\right)^{-1}&\mathbf{O}\\
\mathbf{O}&\left(\bDtwo^{\trans}\bDtwo\right)^{-1}
\end{bmatrix}
\begin{bmatrix}
\bX^{\trans}\\
\bDtwo^{\trans}
\end{bmatrix}
\bOmega
\begin{bmatrix}
\bX&\bDtwo
\end{bmatrix}\\
&=
\begin{bmatrix}
\left(\bX^{\trans}\bX\right)^{-1}&\mathbf{O}\\
\mathbf{O}&\left(\bDtwo^{\trans}\bDtwo\right)^{-1}
\end{bmatrix}
\begin{bmatrix}
\mathbf{O}_2&\mathbf{O}\\
\mathbf{O}&\bDtwo^{\trans}\bOmega\bDtwo
\end{bmatrix}
=
\mathbf{O}_2\oplus(\bT^{-1}\bM),
\end{aligned}
\]
so $\bOmega$ is similar to $\mathbf{O}_2\oplus(\bT^{-1}\bM)$.

For the second statement, suppose $\bT^{-1}\bM\bx=\lambda\bx$ for some
nonzero vector $\bx$, and let $\widehat{\bx}$ be the vector obtained from
$\bx$ by inserting two zeros at the beginning. Then
$\bigl(\mathbf{O}_2\oplus(\bT^{-1}\bM)\bigr)\widehat{\bx}=\lambda\widehat{\bx}$.
Undoing the similarity transformation gives
\[
\bOmega
\begin{bmatrix}
\bX&\bDtwo
\end{bmatrix}
\widehat{\bx}
=
\lambda
\begin{bmatrix}
\bX&\bDtwo
\end{bmatrix}
\widehat{\bx} =\lambda\bDtwo\bx,
\qquad\text{hence}\qquad
\bOmega\bDtwo\bx
=
\lambda \bDtwo\bx.
\]
Since $\bDtwo$ has full column rank and $\bx\neq\mathbf0$, we have
$\bDtwo\bx\neq\mathbf0$, so $\bDtwo\bx$ is indeed an eigenvector of $\bOmega$
corresponding to $\lambda$.
\end{proof}

\begin{proposition}
\label{prop:mos}
The matrix $\bM$ has the form
\[
\bM=
\begin{bmatrix}
8 & 2 & 0 & \cdots & 0\\
2 & 8 & 2 & \ddots & \vdots\\
0 & 2 & \ddots & \ddots & 0\\
\vdots & \ddots & \ddots & 8 & 2\\
0 & \cdots & 0 & 2 & 8
\end{bmatrix}
\]
and is oscillatory.
\end{proposition}

\begin{proof}
Since the columns of $\bDtwo$ are orthogonal to those of $\bX$, we have
$\bQ\bDtwo=\bDtwo$, and hence
$\bM=\bDtwo^{\trans}\bOmega\bDtwo=\bDtwo^{\trans}\bP \bDtwo$.
A direct computation shows that $\bDtwo^{\trans}\bP \bDtwo$ has the stated
tridiagonal form. Since $\bM$ is a Jacobi matrix that is strictly diagonally
dominant, Proposition~\ref{prop:jacobi_oscillatory} implies that $\bM$ is oscillatory.
\end{proof}

\begin{lemma}
\label{lem:tinvtn}
The matrix $\bT^{-1}$ is totally nonnegative.
\end{lemma}

\begin{proof}
By Proposition~\ref{prop:tninv}, it suffices to show that $\bR\bT\bR$ is
totally nonnegative.  Recall that $|\bDtwo|$ denote the matrix obtained from $\bDtwo$
by taking the absolute value of every entry; a direct computation gives
$\bR\bT\bR=|\bDtwo|^{\trans}|\bDtwo|$. By Proposition~\ref{prop:tnprod}, it
is enough to prove that $|\bDtwo|$ is totally nonnegative. Let $J$ be the $n\times n$ 
Jacobi matrix whose diagonal entries are $2$ and whose super- and
sub-diagonal entries are $1$. Its leading principal minors $\delta_k:=\det J[\{1,\ldots,k\}]$ satisfy the recurrence $\delta_k=2\delta_{k-1}-\delta_{k-2}$ with $\delta_0=1$ and $\delta_1=2$, so induction yields $\delta_k=k+1>0$. By Proposition~\ref{prop:osjacobi} and the positive off-diagonal entries, $J$ is
totally nonnegative, and since $|\bDtwo|$ is a submatrix of $J$, it is also
totally nonnegative.
\end{proof}

\begin{corollary}
\label{cor:tinvmos}
The matrix $\bT^{-1}\bM$ is oscillatory.
\end{corollary}

\begin{proof}
By Lemma~\ref{lem:tinvtn}, $\bT^{-1}$ is an invertible totally nonnegative
matrix, while Proposition~\ref{prop:mos} shows that $\bM$ is oscillatory. 
The transpose of an oscillatory matrix is oscillatory: if $\bA$ is
totally nonnegative then so is $\bA^{\trans}$, and if $\bA^{k}$ is
totally positive then so is $(\bA^{\trans})^{k}=(\bA^{k})^{\trans}$.
In particular $\mathbf{M}^{\trans}$ is oscillatory, while
$(\bT^{-1})^{\trans}$ remains invertible and totally nonnegative.
Proposition~\ref{prop:ostnos} therefore implies that the product
$\mathbf{M}^{\trans}(\bT^{-1})^{\trans}$ is oscillatory.
Taking transposes again, it follows that
$\bT^{-1}\mathbf{M}
=\bigl(\mathbf{M}^{\trans}(\bT^{-1})^{\trans}\bigr)^{\trans}
$ is oscillatory.
\end{proof}

\begin{theorem}
\label{thm:main}
The matrix $\bOmega$ is positive semidefinite with nullity $2$, and all of
its positive eigenvalues are simple. Let
$\lambda_1>\cdots>\lambda_{n-2}>0=\lambda_{n-1}=\lambda_n$
be the eigenvalues of $\bOmega$. Then the following statements hold.
\begin{enumerate}
\item
The kernel of $\bOmega$ is spanned by $(1,\ldots,1)^{\trans}$ and
$(1,\ldots,n)^{\trans}$.
\item
For $k=1,\ldots,n-2$, the eigenvector corresponding to $\lambda_k$ is even
whenever $k-1$ is even, and odd whenever $k-1$ is odd.
\item
The eigenvector of $\bOmega$ corresponding to $\lambda_k$ has $k +1$ sign
changes for $k=1,\ldots,n-2$.
\end{enumerate}
\end{theorem}

\begin{proof}
By Corollary~\ref{cor:tinvmos}, the matrix $\bT^{-1}\bM$ is an $(n-2)\times(n-2)$ oscillatory matrix. 
By Theorem~\ref{thm:osspec}, all of its eigenvalues are simple and positive, satisfying
\[
\lambda_1 > \lambda_2 > \cdots > \lambda_{n-2} > 0.
\]
According to Lemma~\ref{lem:asim}, the nonzero eigenvalues of $\bOmega$ coincide with those of $\bT^{-1}\bM$, while its remaining two eigenvalues are zero. Since $\bOmega$ is symmetric, it is positive semidefinite with nullity $2$, and all its positive eigenvalues are simple.

\smallskip\noindent\textit{Statement (1).}
By definition, $\bOmega\bX = \mathbf{O}$. Since $\bX$ has full column rank ($\rank(\bX)=2$), the two columns of $\bX$, namely $(1,\ldots,1)\trans$ and $(1,\ldots,n)\trans$, form a basis for $\ker(\bOmega)$.

\smallskip\noindent\textit{Statement (2).}
Let $\bx_k$ be an   eigenvector of $\bT^{-1}\bM$ corresponding to $\lambda_k$ for $k=1,\ldots,n-2$.
The $j$th column of $\bDtwo$ is supported on rows $j,j+1,j+2$, so reversing rows sends it to the $(n-1-j)$th column. Thus $\bZ_n\bDtwo=\bDtwo\bZ_{n-2}$, where subscripts on $\bZ$ denote the dimensions, and therefore $\bZ_{n-2}\bT\bZ_{n-2}=\bT$. The matrix $\bM$ in Proposition~\ref{prop:mos} is symmetric Toeplitz, hence centrosymmetric.
Note that $\bZ_{n-2} \bT^{-1}\bM \bZ_{n-2} = (\bZ_{n-2} \bT^{-1}\bZ_{n-2})(\bZ_{n-2} \bM \bZ_{n-2}) = \bT^{-1}\bM$, which implies that $\bT^{-1}\bM$ is centrosymmetric, so by Proposition~\ref{prop:csym} and the fact that all of its eigenvalues are simple (Corollary \ref{cor:tinvmos} and Theorem~\ref{thm:osspec}), every eigenvector $\bx_k$ is either even or odd.
By Theorem~\ref{thm:osspec}, the   eigenvector $\bx_k$ of the oscillatory matrix $\bT^{-1}\bM$ has exactly $k-1$ sign changes, i.e., $S({\bx_k}) = k-1$.
Since an even vector with a well-defined $S$ always possesses an even number of sign changes and an odd vector with a well-defined $S$ possesses an odd number of sign changes, $\bx_k$ must be even when $k-1$ is even, and odd when $k-1$ is odd.
Then $\bpsi_k=\bDtwo \bx_k$ has the same parity as $\bx_k$.

\smallskip\noindent\textit{Statement (3).}

 Let $\bpsi_k=\bDtwo\bx_k$ be the corresponding eigenvector of $\bOmega$.  We use the subscripts of $\mathbf{R}$ to indicate the dimensions of sign-alternating diagonal matrices. Since $\bx_k\neq\mathbf0$ is an eigenvector, and $\bR_{n-2}$ is invertible, it follows that $\bR_{n-2}\bx_k\neq\mathbf 0$.

\smallskip\noindent We start by establishing the lower bound $S^{-}(\bpsi_k)\ge k+1$.
Since $\bR_n\bDtwo=|\bDtwo|\bR_{n-2}$, Corollary~\ref{cor:D2vm} applied to
$\bu=\bR_{n-2}\bx_k\ne\mathbf0$ gives
\[
S^{+}(\bR_n\bpsi_k)=S^{+}\bigl(|\bDtwo|\bR_{n-2}\bx_k\bigr)\le S^{+}(\bR_{n-2}\bx_k).
\]
Applying Corollary~\ref{cor:signflip} to $\bpsi_k$ (with $\ell=n$) and to $\bx_k$ (with
$\ell=n-2$) gives $S^{+}(\bR_n\bpsi_k)=(n-1)-S^{-}(\bpsi_k)$ and
$S^{+}(\bR_{n-2}\bx_k)=(n-3)-S^{-}(\bx_k)$. Combining these with the inequality above,
\[
(n-1)-S^{-}(\bpsi_k)\ \le\ (n-3)-S^{-}(\bx_k),
\]
so, using $S^{-}(\bx_k)=S(\bx_k)=k-1$ (Theorem~\ref{thm:osspec}), we obtain
\[
S^{-}(\bpsi_k)\ \ge\ S^{-}(\bx_k)+2\ =\ k+1.
\]

\smallskip\noindent Next we derive the upper bound $S^{+}(\bpsi_k)\le k+1$.
Applying Lemma~\ref{lem:weakTN} to the totally nonnegative matrix $|\bDtwo|^{\trans}$ and
the vector $\bR_n\bpsi_k$, and using $\bR_{n-2}\bDtwo^{\trans}=|\bDtwo|^{\trans}\bR_n$,
\[
S^{-}\bigl(|\bDtwo|^{\trans}\bR_n\bpsi_k\bigr)
=S^{-}(\bR_{n-2}\bDtwo^{\trans}\bpsi_k)
=S^{-}(\bR_{n-2}\bT\bx_k)
\ \le\ S^{-}(\bR_n\bpsi_k).
\]
Applying Lemma~\ref{lem:signflip} to $\bT\bx_k$ (with $\ell=n-2$) and to $\bpsi_k$ (with
$\ell=n$) converts this into
\[
(n-3)-S^{+}(\bT\bx_k)\ \le\ (n-1)-S^{+}(\bpsi_k),
\qquad\text{i.e.}\qquad
S^{+}(\bpsi_k)\ \le\ S^{+}(\bT\bx_k)+2.
\]
Since $\bT\bx_k=\lambda_k^{-1}\mathbf{M}\bx_k$  with $\lambda_k>0$ and, by Proposition~\ref{prop:mos},
$\mathbf{M}$ is oscillatory, hence invertible totally nonnegative, Lemma~\ref{lem:InTN}
gives $S^{+}(\mathbf{M}\bx_k)\le S^{+}(\bx_k)=S(\bx_k)=k-1$ with no further hypothesis on
$\bx_k$. Hence
\[
S^{+}(\bpsi_k)\ \le\ S^{+}(\mathbf{M}\bx_k)+2\ \le\ S^{+}(\bx_k)+2\ =\ k+1.
\]

\smallskip\noindent Finally, since
$S^{-}\le S^{+}$ holds for every vector, the two bounds give
\[
k+1\ \le\ S^{-}(\bpsi_k)\ \le\ S^{+}(\bpsi_k)\ \le\ k+1,
\]
forcing $S^{-}(\bpsi_k)=S^{+}(\bpsi_k)=k+1$, i.e.\ $S(\bpsi_k)$ is well defined and
equals $k+1$.

\end{proof}

\begin{remark}\label{rem:DR}
The penalty matrix $\bDtwo\bW^{-1}\bDtwo\trans$ of \cite{DR75}, with $\bW$ the Gram matrix
of the hat functions,  is, up to a scale, the Moore--Penrose inverse of $\bOmega$; moreover, $\bT^{-1}\bM$ is,
up to a scale, the matrix $(\bDtwo\trans\bDtwo)^{-1}\bW$ analyzed there. Hence the $k$th
largest eigenvalue of $\bOmega$ and the $k$th smallest positive eigenvalue of the penalty
matrix share the same eigenvector, so the decreasing ordering in Theorem~\ref{thm:main}
reproduces the increasing ordering of \cite{DR75}. In particular, the $k+1$ sign changes in
Theorem~\ref{thm:main}(3) agree with \cite[Theorem~7]{DR75}, where the first two eigenvectors span the null space.
The increasing ordering of eigenvalues in \cite{DR75} parallels that in classical Sturm--Liouville theory, where the eigenfunction for the $k$th smallest eigenvalue has exactly $k-1$ interior zeros \cite{Sturm1836,Zettl2005}.
\end{remark}

\section{Parity of the Spline Eigenbasis Functions}
\label{sec:spline-parity}

We now transfer the parity property of the eigenvectors of $\bOmega$ to the
continuous spline functions they represent.  We may, without loss of generality, apply an affine rescaling to the knots so that they lie symmetrically in $[-1,1]$. Specifically, let $x_i=(2i-n-1)/(n-1), \; i=1,\ldots,n,$  and define $\mathbf{s}=(x_1,\ldots,x_n)\trans$ and $\bX'=[\bone,\mathbf{s}]$. 

The space $\operatorname{span}\{\bone,\mathbf{s}\}$  remains identical to $\operatorname{span}\{\bone,(1,\ldots,n)\trans\}$, and hence $\operatorname{span}(\bX')=\operatorname{span}(\bX)$ and $\bQ$ is invariant under rescaling. Moreover, $\mathbf{s}$ satisfies $\mathbf{Z}\mathbf{s}=-\mathbf{s}$ and $\mathbf{1}^{\trans}\mathbf{s}=0$, so that $\mathbf{X'}^{\trans}\mathbf{X'}=\operatorname{diag}\left(n,\mathbf{s}^{\trans}\mathbf{s}\right).$
The corresponding matrices $\bP$ and $\bOmega$ are changed only by an overall scalar factor, so their eigenvectors remain unchanged. For simplicity, we continue to denote the transformed matrices by $\bX,\bP,\bQ,\bOmega$. 

Concretely, since $|x_i-x_j|=\dfrac{2}{n-1}|i-j|$, the rescaled kernel matrix equals $h^3$ times the one used in Sections~\ref{sec:prelim}--\ref{sec:main}, with $h=\left(\dfrac{2}{n-1}\right)$, so $\bOmega$ here is likewise $h^3$ times its earlier counterpart. In particular, the specific tridiagonal entries $8$ and $2$ recorded for $\bM$ in Proposition~\ref{prop:mos} refer to the unrescaled knots $1,\ldots,n$ of Sections~\ref{sec:prelim}--\ref{sec:main} and become $8h^3,2h^3$ under the present rescaling; this affects neither the total nonnegativity/oscillatory property of $\bM$ nor any sign-variation or parity statement used below, all of which depend only on eigenvectors and signs of entries, and are therefore invariant under multiplication by the positive scalar $h^3$.

  Define the kernel evaluation vector $\bphi(x)=(\phi(x-x_1),\ldots,\phi(x-x_n))\trans$ with $\phi(r)=|r|^3$, which satisfies $\bphi(-x)=\bZ\bphi(x)$. We then follow \cite{TH2018} and write $\bOmega=\sum_{k=1}^{n-2}\lambda_k\bpsi_k\bpsi_k\trans$ with unit eigenvectors $\bpsi_k$ (as in Theorem~\ref{thm:main}). Define the spline basis $\{f_1,\ldots,f_n\}$ with $f_1(x)=1$, $f_2(x)=x$, and for $k=1,\ldots,n-2$,
\begin{equation}
\label{eq:fk_def}
f_{k+2}(x)=\lambda_k^{-1}\bpsi_k\trans\left[\bphi(x)-\bP\bX(\bX\trans\bX)^{-1}\begin{bmatrix}1\\ x\end{bmatrix}\right],
\end{equation}
which   subtracts from each entry of $\bphi(x)$ its least-squares linear fit
at the knots, so $f_{k+2}(x)-\lambda_k^{-1}\bpsi_k\trans\bphi(x)$ is affine in $x$. Using $\bone\trans\mathbf{s}=0$, \eqref{eq:fk_def} simplifies to
\begin{equation}
\label{eq:fk_simp}
f_{k+2}(x)=\lambda_k^{-1}\bpsi_k\trans\bphi(x)
-\lambda_k^{-1}\left(\frac{\bpsi_k\trans\bP\bone}{n}
+\frac{\bpsi_k\trans\bP\mathbf{s}}{\mathbf{s}\trans\mathbf{s}}\,x\right).
\end{equation}

The next lemma shows that $f_{k+2}$ interpolates the eigenvector
$\bpsi_k$ of $\bOmega$ at the knots; that $f_{k+2}$ is a natural cubic
spline is verified in Proposition~\ref{prop:fpp-nodal}. 

\begin{lemma}
\label{lem:interp}
For $k=1,\dots,n-2$ and $i=1,\dots,n$,
\[
f_{k+2}(x_i)=\psi_{k,i}, 
\]
where $\psi_{k,i}$ is the $i$th entry of $\bpsi_k$. 
\end{lemma}

\begin{proof}
Let $\tilde{f}_k(x):=\bpsi_k^{\trans}\bigl(\bphi(x)- {\bP} \bX(\bX^{\trans} \bX)^{-1}[1,x]^{\trans}\bigr)$.
The nodal vector of $\tilde{f}_k$ is $\bQ\bP\bpsi_k$.
Lemma~\ref{lem:asim} and Proposition~\ref{prop:D2_part} give $\bQ\bpsi_k=\bpsi_k$, so
$\bQ\bP\bpsi_k=\bQ\bP \bQ\bpsi_k=\bOmega\bpsi_k=\lambda_k\bpsi_k$.
Thus $f_{k+2}=\lambda_k^{-1}\tilde{f}_k$ interpolates $\bpsi_k$.
\end{proof}

\begin{theorem}
\label{thm:spline-parity}
For $k=1,\ldots,n-2$, $f_{k+2}$ is even when $k-1$ is even, and odd when $k-1$ is odd.
\end{theorem}

\begin{proof}
Let $g_k(x):=\bpsi_k\trans\bphi(x)$. Since $\bphi(-x)=\bZ\bphi(x)$ and $\bZ\bpsi_k=\pm\bpsi_k$ according as $k-1$ is even or odd (Theorem~\ref{thm:main}(2)), we have $g_k(-x)=(\bZ\bpsi_k)\trans\bphi(x)=\pm g_k(x)$, so $g_k$ inherits the parity of $\bpsi_k$.

Next, consider the correction term $c_k(x):=\frac{1}{n}\bpsi_k\trans\bP\bone+\frac{1}{\mathbf{s}\trans\mathbf{s}}(\bpsi_k\trans\bP\mathbf{s})\,x$.  Centrosymmetry of $\bP$ ($\bZ\bP=\bP\bZ$) implies that $\bP\bpsi_k$ shares the parity of $\bpsi_k$, so one of the inner products vanishes by orthogonality of opposite-parity vectors.

    When $k-1$ is even, $\bP\bpsi_k$ is even. Since $\mathbf{s}$ is odd, $\bpsi_k\trans\bP\mathbf{s}=0$, reducing $c_k(x)$ to a constant (even).

    When $k-1$ is odd, $\bP\bpsi_k$ is odd. Since $\bone$ is even, $\bpsi_k\trans\bP\bone=0$, reducing $c_k(x)$ to a multiple of $x$ (odd).

\noindent In both cases, $c_k(x)$ has the same parity as $g_k(x)$, and thus $f_{k+2}=\lambda_k^{-1}(g_k-c_k)$ has the claimed parity.
\end{proof}

\begin{remark}
By Remark~\ref{rem:DR}, $\bOmega$ and the roughness-penalty matrix
$\bDtwo\bW^{-1}\bDtwo\trans$ describe the same spectral modes from two complementary
viewpoints: from the projected kernel and from the roughness penalty, respectively.
The kernel viewpoint expresses the corresponding spline functions explicitly through the kernel evaluation vector
$\bphi(x)$ in \eqref{eq:fk_def}, so that the symmetry $\bphi(-x)=\bZ\bphi(x)$ carries the
parity of the eigenvectors directly over to the functions, as in
Theorem~\ref{thm:spline-parity}. The parity structure itself, which relies on the symmetry of the equally spaced knots, was not considered in \cite{DR75}.
\end{remark}

\section{Oscillation and Zeros of the Spline Eigenbasis Functions}
\label{sec:spline-osc}

In Section~\ref{sec:spline-parity}, Theorem~\ref{thm:spline-parity} establishes the parity structure of the spline eigenfunctions $f_{k+2}(x)$. Here, we show that the sign-change count $S(\bpsi_k) = k+1$ of the discrete eigenvector transfers seamlessly to the continuous function $f_{k+2}(x)$, showing that $f_{k+2}$ has exactly $k+1$ simple zeros in $(-1,1)$.

Fix $k \in \{1,\dots,n-2\}$, and let $\lambda = \lambda_k > 0$, $\bpsi = \bpsi_k$, $\bx =\bx_k$, and $f = f_{k+2}$ be defined as in Section~\ref{sec:spline-parity}. Recall that $S(\bpsi) = k+1$ and $S(\bx) = k-1$.

\begin{proposition}
\label{prop:fpp-nodal}
The function $f(x)$ is a natural cubic spline in $C^2[-1,1]$ with knot values $f(x_i) = \psi_i$  and boundary conditions $f''(\pm 1) = 0$. Let $\balpha = (f''(x_2), \dots, f''(x_{n-1}))^{\trans}$. Then
\begin{equation}
  \balpha = \frac{12h}{\lambda}\,\bx,
\end{equation}
where $h = 2/(n-1)$. Consequently, $S(\balpha) = S(\bx) = k-1$.
\end{proposition}

\begin{proof}
Since $\bOmega$ is symmetric with $\bOmega\bX=\mathbf 0$ while $\bOmega\bpsi=\lambda\bpsi$ for $\lambda\ne0$, and eigenvectors of a symmetric matrix for distinct eigenvalues are orthogonal, we have $\bX^{\trans}\bpsi=\mathbf 0$, i.e.\ $\sum_i\psi_i=0$ and $\sum_i\psi_i x_i=0$. 
Each $|t-x_j|^3$ is a cubic polynomial in $t$ on every $[x_i,x_{i+1}]$ and lies in $C^2(\mathbb R)$ with
$(|r|^3)''=6|r|$, so $f$ is a $C^2$ piecewise cubic with knots $x_i$, and
$f''(t)=\frac{6}{\lambda}\sum_j\psi_j|t-x_j|$. The moment conditions then give
$f''(x_n)=\frac{6}{\lambda}\sum_j\psi_j(x_n-x_j)=0$ and
$f''(x_1)=\frac{6}{\lambda}\sum_j\psi_j(x_j-x_1)=0$, i.e.\ $f''(\pm1)=0$.

Next, on each knot interval $[x_i, x_{i+1}]$, setting $\tau = (t - x_i)/h \in [0,1]$, the unique cubic polynomial matching the endpoint values $ f(x_i) = \psi_i$ (Lemma~\ref{lem:interp}), $\psi_{i+1}$ and second derivatives $\alpha_i = f''(x_i)$, $\alpha_{i+1}$ is
\[
  f(t) = (1-\tau)\psi_i + \tau \psi_{i+1} + \frac{h^2}{6}\Bigl( ((1-\tau)^3 - (1-\tau)) \alpha_i + (\tau^3 - \tau) \alpha_{i+1} \Bigr).
\]
Differentiating with respect to $t$ gives $h f'(x_i^+) = (\psi_{i+1}-\psi_i) - \frac{h^2}{6}(2 \alpha_i + \alpha_{i+1})$ at $\tau=0$, and $h f'(x_i^-) = (\psi_i-\psi_{i-1}) + \frac{h^2}{6}( \alpha_{i-1} + 2 \alpha_i)$ at $\tau=1$ on $[x_{i-1}, x_i]$. Equating $f'(x_i^+) = f'(x_i^-)$ yields
\[
  \alpha_{i-1} + 4 \alpha_i + \alpha_{i+1} = \frac{6}{h^2}(\psi_{i-1} - 2\psi_i + \psi_{i+1}), \qquad i=2,\dots,n-1.
\]
Let $\bK$ be a $(n-2)\times(n-2)$ tridiagonal matrix with main diagonal $4$ and off-diagonals $1$. The system above reads $\bK \balpha = \frac{6}{h^2} \bDtwo^{\trans} \bpsi$. 

Since $\bpsi = \bDtwo \bx$ and $\bT = \bDtwo^{\trans} \bDtwo$, we have $\bDtwo^{\trans} \bpsi = \bDtwo^{\trans} \bDtwo \bx = \bT \bx$. Under uniform spacing, $\bM = 2h^3 \bK$. Combining this with $\bT^{-1} \bM \bx = \lambda \bx$ yields $\bT \bx = \frac{1}{\lambda} \bM \bx = \frac{2h^3}{\lambda} \bK \bx$. Substituting $\bT \bx$ into $\bK \balpha = \frac{6}{h^2} \bT \bx$ gives $\bK \balpha = \frac{12h}{\lambda} \bK\bx$. Since $\bK$ is strictly diagonally dominant, it is invertible, yielding $\balpha = \frac{12h}{\lambda} \bx$. Then $\balpha$ directly inherits the sign pattern of $\bx$, so $S(\balpha) = S(\bx) = k-1$.
\end{proof}

\begin{lemma}
\label{lem:pwl-zeros}
Let $\bu = (u_1, \dots, u_\ell)^{\trans} \in \mathbb{R}^\ell$ ($\ell \ge 2$) be a vector for which $S(\bu)$ exists in the sense of Definition~\ref{def:S_pm}. Let $g \in C[x_1, x_\ell]$ be the piecewise linear interpolant with $g(x_i) = u_i$ for $i=1,\dots,\ell$ and $x_1<\dots<x_\ell$. Then $g$ has no zero interval, and its total number of distinct zeros on $[x_1, x_\ell]$, denoted by $\zeta(g)$, satisfies
\[
  \zeta(g) = S(\bu).
\]
\end{lemma}

\begin{proof}
Since $S(\bu)$ exists, $u_1 u_\ell \neq 0$ and $\bu$ contains no two consecutive zeros. On each interval $[x_i, x_{i+1}]$, the linear function $g$ vanishes identically if and only if $u_i = u_{i+1} = 0$, which is excluded. Thus all zeros of $g$ are isolated. A zero occurs at a knot $x_i$ if and only if $u_i = 0$ (yielding $N_0(\bu)$ zeros, where $N_0(\bu)$ is the number of zero entries), or in an open interval $(x_i, x_{i+1})$ if and only if $u_i u_{i+1} < 0$ (yielding $N_\pm(\bu)$ zeros, where $N_\pm(\bu)$ is the number of strict sign changes). Since $u_1, u_\ell \neq 0$ and no two zeros are consecutive, all zeros are isolated . A zero occurs at a knot $x_i$ precisely when $u_i=0$ $(N_0(\mathbf{u})$ such zeros) and in an open interval $(x_i,x_{i+1})$ precisely when $u_i u_{i+1}<0$ $(N_\pm(\mathbf{u})$ such zeros). By the characterization following Definition~\ref{def:S_pm},  each zero entry satisfies $u_{i-1}u_{i+1}<0$, so each knot zero corresponds to one sign change in the compressed sequence. Hence $\zeta(g) = N_0(\bu) + N_\pm(\bu) = S(\bu)$.
\end{proof}

When $n=3$, $\balpha=(\alpha_1)$ with $\alpha_1\ne0$, and $f''$ is linear on
   $[-1,0]$ and $[0,1]$ with $f''(\pm1)=0$; hence its only zeros are $\pm1$ and
   $\zeta(f'')=2=k+1$.

\begin{proposition}
\label{prop:fpp-zeros}
The second derivative $f''$ does not vanish identically on any subinterval of $[-1,1]$. It has isolated zeros, and its total number of zeros on $[-1,1]$ is $\zeta(f'') = k+1$ (including the endpoints $\pm 1$).
\end{proposition}

\begin{proof}
Since \(f\) is a natural cubic spline, \(f''\) is continuous and piecewise linear on \([-1,1]\) with \(f''(\pm1)=0\). By Proposition~\ref{prop:fpp-nodal}, \(S(\balpha)=k-1\) with nonzero endpoints and no consecutive zeros, so Lemma~\ref{lem:pwl-zeros} yields exactly \(k-1\) zeros of \(f''\) in \((x_2,x_{n-1})\). The end intervals \([x_1,x_2]\) and \([x_{n-1},x_n]\) are linear and connect an endpoint zero to a nonzero interior node, contributing no further zeros. Adding the two endpoint zeros therefore gives
\[
\zeta(f'')=(k-1)+2=k+1.
\]
\end{proof}

\begin{theorem}
\label{thm:spline-oscillation}
For each $k=1,\dots,n-2$, the spline eigenfunction $f_{k+2}$ satisfies $f_{k+2}(\pm 1) \neq 0$ and has exactly $k+1$ distinct, simple, sign-changing zeros in $(-1,1)$. Consequently, $\zeta(f_{k+2}) = k+1$.
\end{theorem}

\begin{proof}
Since $f(x_1) = \psi_1 \neq 0$ and $f(x_n) = \psi_n \neq 0$, all zeros of $f$ lie in $(-1,1)$. Since $f$ does not vanish on any subinterval (which would force $f'' \equiv 0$ on that interval, contradicting Proposition~\ref{prop:fpp-zeros}), its zeros are isolated. The vector of knot values  $\psi$ has $S(\bpsi) = k+1$ sign changes, so by the Intermediate Value Theorem, $f$ has at least $k+1$ sign-changing zeros, i.e., $\zeta(f) \ge k+1$.

For the upper bound, suppose $f$ has $\zeta(f)$ distinct zeros $\rho_1 < \dots < \rho_\zeta$ in $(-1,1)$, and let $\mu(f) \ge 0$ be the number of zeros where $f' = 0$ (double zeros). By Rolle's Theorem, between any two adjacent zeros there exists an interior extremum, giving $\zeta(f)-1$ extrema. Combining these with the $\mu(f)$ double zeros yields $q = \zeta(f) - 1 + \mu(f)$ distinct zeros of $f'$ in $(-1,1)$. Applying Rolle's Theorem again to $f'$ produces at least $q-1$ interior zeros of $f''$. Together with $f''(\pm 1) = 0$, we get $\zeta(f'') \ge q + 1 = \zeta(f) + \mu(f)$.

By Proposition~\ref{prop:fpp-zeros}, $\zeta(f'') = k+1$. Thus $k+1 \ge \zeta(f) + \mu(f)$. Combining with $\zeta(f) \ge k+1$ yields $\zeta(f) = k+1$ and $\mu(f) = 0$. Hence, $f_{k+2}$ has exactly $k+1$ zeros, all of which are simple and sign-changing.
\end{proof}

\section*{Acknowledgments}
This work was partially supported by the National Science and Technology Council, Taiwan, under Grant Nos.~MOST110-2118-M029-006-MY2,  NSTC112-2118-M-008-001-MY3, NSTC113-2115-M-110-010-MY3, and NSTC114-2111-M008-037.


\end{document}